\documentclass{article}
\usepackage[english]{babel}										
\usepackage[utf8]{inputenc} 
\usepackage[T1]{fontenc}										
\usepackage{amsmath,amsfonts,amssymb,amsthm,cancel,siunitx,
calculator,calc,mathtools,empheq,latexsym,mathtools}

\usepackage{subfig,epsfig,tikz,float}		            
\usepackage{booktabs,multicol,multirow,tabularx,array}          

\usepackage[utf8]{inputenc} 
\usepackage[T1]{fontenc}    
\usepackage{hyperref}       
\usepackage{url}            
\usepackage{booktabs}       
\usepackage{amsfonts}       
\usepackage{nicefrac}       
\usepackage{microtype}      
\usepackage{bm}
\usepackage{graphicx} 
\usepackage{algorithm}      
\usepackage[noend]{algpseudocode} 
\usepackage{caption}
\usepackage{soul}
\usepackage[pagewise]{lineno}
\usepackage{lmodern}
\usepackage{array}          
\usepackage{booktabs}       
\usepackage{pbox}           
\usepackage{authblk}
\usepackage{amsmath} 
\usepackage{listings}
\usepackage{tikz}
\usepackage{mathtools}
\usepackage{float}
\usepackage{subfig}
\usepackage{todonotes}
\usepackage{amsthm,comment}
\usetikzlibrary{shapes.geometric, arrows}
\newtheorem{theorem}{Theorem}
\newtheorem{Lemma}{Lemma}

\title{Threshold dynamics in a within-host infection model with Crowley–Martin functional response considering periodic effects}

\author[1,2]{Ibrahim Nali}
\author[2]{Attila Dénes}

\affil[1]{Bolyai Institute, University of Szeged, Szeged 6720, Hungary}
\affil[2]{National Laboratory for Health Security, Bolyai Institute, University of Szeged, Szeged 6720, Hungary}

\affil[ ]{\textit{Correspondence:} \href{mailto:ibrahim.nali@usmba.ac.ma}{ibrahim.nali@usmba.ac.ma},}

\affil[ ]{\textsuperscript{*} These authors contributed equally to this work.}
\date{}
\begin{document}

\maketitle
\abstract
{We present a mathematical model for within-host viral infections that incorporates the Crowley–Martin functional response, focusing on the dynamics influenced by periodic effects. This study establishes key properties of the model, including the existence, uniqueness, positivity, and boundedness of periodic orbits within the non-autonomous system. We demonstrate that the global dynamics are governed by the basic reproduction number, denoted as $\mathcal{R}_0$, which is calculated using the spectral radius of an integral operator. Our findings reveal that $\mathcal{R}_0$ serves as a threshold parameter: when $\mathcal{R}_0 < 1$, the virus-free periodic solution is globally asymptotically stable, indicating that the infection will die out. Conversely, if $\mathcal{R}_0 > 1$, at least one positive periodic solution exists, and the disease persists uniformly, with trajectories converging to a limit cycle. Additionally, we provide numerical simulations that support and illustrate our theoretical results, enhancing the understanding of threshold dynamics in within-host infection models.}

\section{Introduction}
Recently, there has been a growing emphasis on studying epidemic models from a within-host perspective to gain deeper insights into viral infections. This area of research has gained traction not only in time-invariant models but also in those incorporating periodic effects \cite{Zhangwang, Badersaad}. The intricate dynamics of viral infections within host organisms present a compelling area of study, particularly when examining the interplay between biological and environmental factors that influence infection outcomes. These factors, including periodic treatments and seasonal variations, significantly impact viral replication and host interactions, shaping the progression and control of infections. For instance, circadian rhythms regulate physiological processes such as immune responses, while seasonal variations in contact rates or vaccination programs also influence disease transmission \cite{Hajji,Koelle, Mazzoccoli,Seittos,Steindrof}.

The basic reproduction number, denoted as $\mathcal{R}_0$, is a critical parameter in disease transmission studies, representing the expected number of secondary cases generated by a typical infectious individual in a completely susceptible population. In classical epidemic models, $\mathcal{R}_0 < 1$ implies disease extinction, whereas $\mathcal{R}_0 > 1$ indicates potential persistence. While traditional approaches have effectively calculated $\mathcal{R}_0$ for various infectious diseases \cite{Diekmann}, its computation in periodic epidemic models remains a challenge due to fluctuations in disease transmission driven by seasonal or treatment-related factors \cite{Bacer}. Accurate definitions and computations of $\mathcal{R}_0$ in periodic contexts are essential, as time-averaged autonomous systems may not accurately reflect infection risks \cite{Farrington}. Recent studies have established general definitions of $\mathcal{R}_0$ in periodic environments, highlighting its role as a threshold parameter for local stability of disease-free periodic solutions and understanding global dynamics under specific conditions.

The Crowley--Martin functional response (CMFR) offers a robust framework for modeling infection dynamics in viral systems by addressing key limitations of bilinear incidence functions. These functions often assume constant interaction rates between host cells and pathogens, an assumption that breaks down in biological scenarios involving large viral populations due to competition for resources, leading to reduced interaction efficiency \cite{Yuyang, Li}. The CMFR, originally developed in the context of predator-prey dynamics, accounts for competition effects and dynamic interaction rates, making it particularly suited for in-host models of viral infections \cite{Paul, Crowley}. Defined as 
$$
\frac{\beta T V}{(1 + aT)(1 + bV)},
$$ 
where $T$ represents the density of susceptible cells, $V$ denotes the population of free virus particles, $\beta$ is the maximum infection rate, and $a, b > 0$ capture the effects of capture rate and interference, respectively, the CMFR effectively captures saturation phenomena in viral replication driven by limited target cell availability \cite{Cosner}. Analogous to how higher predator densities reduce feeding rates in predator-prey systems due to spatial factors and competition \cite{DeAngelis, Debnath}, viral dynamics exhibit similar interference effects, where increasing viral loads diminish infection efficiency. This refined perspective provides critical insights into the interplay between viruses and immune responses, offering a more biologically realistic depiction of in-host viral dynamics.

Building on our previous work, which employed an autonomous ordinary differential equation model to study the dynamics of Usutu virus \cite{Nali}, we now extend our analysis by incorporating periodic factors to account for seasonality within host-virus interactions. While environmental factors like climatic changes and vector dynamics are known to influence transmission \cite{Martinez, Hridoy, Hou}, it is equally critical to address seasonality within the host's physiological and treatment dynamics \cite{MA,Guerrero,Yang,Bacaer}. For instance, seasonal variations in immune responses and the timing of regular treatment doses can significantly impact viral replication and clearance rates. The Usutu virus demonstrates heightened virulence in specific seasons, with increased transmission during the dry season \cite{Patterson}. Furthermore, periodic medical interventions, such as treatments or vaccinations, can modulate infection outcomes, emphasizing the necessity of integrating seasonality into in-host models. Motivated by these insights, we generalize our previous autonomous model by incorporating the Crowley–Martin functional response within a periodic framework, enabling a more nuanced understanding of the interplay between seasonal host factors and viral dynamics.

This paper aims to develop a mathematical framework that integrates periodic effects into a within-host viral interaction model. Specifically, we explore the global dynamics with respect to the basic reproduction number $\mathcal{R}_0$, calculated as the spectral radius of an integral operator. We rigorously establish the existence, uniqueness, positivity, and boundedness of periodic solutions to ensure model robustness and applicability. The global stability of the disease-free periodic solution is proven for $\mathcal{R}_0 < 1$, while the persistence of infection is demonstrated for $\mathcal{R}_0 > 1$, indicating potential long-term viral presence. To complement these theoretical results, we provide insightful numerical simulations that visualize the complex dynamics of the system, offering valuable tools for understanding and applying our model in real-world contexts.
\section{Model derivation}
Building on the foundational work of Heitzman-Breen et al.~\cite{heitzman}, this model incorporates a Crowley--Martin type functional response to capture nonlinearities in the infection process. The framework divides the total cell population into three compartments:

\begin{itemize}
    \item \textbf{Healthy target cells ($T$):} Cells susceptible to infection.
    \item \textbf{Exposed cells ($E$):} Cells that have been infected but are not yet infectious.
    \item \textbf{Infected cells ($I$):} Cells actively producing virus particles.
\end{itemize}

Additionally, the dynamics of free virus particles ($V$) are explicitly modeled. To account for key processes within the host, time-dependent parameters are introduced to represent periodic influences arising from environmental, physiological, or therapeutic factors. 

The cell birth rate, $\mu(t)$, and death rate, $d(t)$, are modeled as bounded, continuous, and $\omega$-periodic functions. This periodicity reflects cyclic behaviors, such as circadian rhythms or seasonal variations, that influence cellular turnover. Similarly, the transmission rate $\beta(t)$ is also $\omega$-periodic, representing variations in the efficiency of infection driven by time-varying conditions like immune response fluctuations or external environmental changes.

The infection process is described by a Crowley--Martin functional response:
\[
\frac{\beta(t) T V}{(1 + C_1 T)(1 + C_2 V)},
\]
which accounts for saturation effects due to limited target cells and competitive interactions among virus particles. 

The dynamics of the system are governed by the following differential equations:
\begin{equation}
\label{1}
\begin{aligned}
    \frac{d T}{d t} &= \mu(t) - \frac{\beta(t) T V}{(1 + C_1 T)(1 + C_2 V)} - d(t) T, \\
    \frac{d E}{d t} &= \frac{\beta(t) T V}{(1 + C_1 T)(1 + C_2 V)} - k E - d(t) E, \\
    \frac{d I}{d t} &= k E - \delta I - d(t) I, \\
    \frac{d V}{d t} &= p I - c V.
\end{aligned}
\end{equation}

The parameters used in the model and their roles are summarized in Table~\ref{tab:parameters}.

\begin{table}[h!]
\centering
\begin{tabularx}{\textwidth}{ll}
\toprule
\textbf{Parameter} & \textbf{Definition} \\
\midrule
$\mu(t)$ & $\omega$-periodic cell birth rate  \\
$d(t)$ & $\omega$-periodic cell death rate  \\
$\beta(t)$ & $\omega$-periodic infection rate  \\
$k$ & Transition rate from exposed to infectious  \\
$\delta$ & Disease-induced death rate of infected cells  \\
$p$ & Production rate of virus particles \\
$c$ & Clearance rate of free virus particles  \\
$C_1$ & Saturation effect of target cells on infection  \\
$C_2$ & Saturation effect of virus particles on infection \\
\bottomrule
\end{tabularx}
\caption{Model parameters.}
\label{tab:parameters}
\end{table}

This time-dependent formulation extends the original time-constant model \cite{Nali} by incorporating periodic effects, providing a versatile framework for analyzing infection dynamics under realistic conditions. The inclusion of periodic parameters enables the model to capture oscillatory behaviors that arise from external and internal factors, such as seasonal influences or therapeutic interventions. 
\section{Basic properties}
In this section, we will study some basic properties of model \eqref{1}. First, in the following lemma, we will show boundedness of solutions. 
\begin{Lemma}
\label{Latr}
Assume the initial conditions for the model \eqref{1} as $T(0)>0, E(0)>0, I(0)>0$, and $V(0)>0$ the solutions $T(t),E(t),I(t)$, and $ V(t) $ remain positive for all $t>0$ and are uniformly bounded.
\end{Lemma}

\begin{proof}
Let $\kappa$ be defined as $\kappa=\sup \{t>0: T(t)>0, E(t)>0, I(t)>0, V(t)>0\}$. It is evident that $\kappa>0$. Now, consider the scenarios: either $\kappa=\infty$ or $\kappa<\infty$. In the former case, the solutions trivially remain positive. In the latter case, let us assume the opposite, i.e.\ at least one of $T(\kappa)$, $E(\kappa)$, $I(\kappa)$, or $V(\kappa)$ is zero. 

From system \eqref{1} we obtain
\begin{align*}
& \frac{d T}{d t} \geq -\frac{\beta(t) T V}{(1+C_1 T)(1+C_2 V)}- d(t) T, \\
& \frac{d E}{d t} \geq -k E - d(t) E, \\
& \frac{d I}{d t} \geq -\delta I - d(t) I, \\
& \frac{d V}{d t} \geq -c V,
\end{align*}
which imply the estimates
\begin{align*}
& T(\kappa) \geq T_0 \exp \left\{-\int_0^\kappa\left( \frac{\beta(u) V(u)}{{(1+C_1 T(u))(1+C_2 V(u))}} - d(u)\right) \,du\right\}>0, \\
& E(\kappa) \geq E_0 \exp \left\{-\int_0^\kappa\left(k + d(u)\right) \,du\right\}>0, \\
& I(\kappa) \geq I_0 \exp \left\{-\int_0^\kappa \delta +d(u) \,du\right\}>0, \\
& V(\kappa) \geq V_0 \exp \left\{-c \kappa\right\}>0.
\end{align*}

Thus, $T(\kappa)>0$, $E(\kappa)>0$, $I(\kappa)>0$, and $V(\kappa)>0 \Rightarrow T(t) >0$, $E(t)>0$, $I(t)>0$, and $V(t) >0$ for any $t>0$. This contradicts our assumption that $T(\kappa)=0$, $E(\kappa)$, $I(\kappa)$, and $V(\kappa)=0$. Therefore, $T(t) >0$, $E(t)>0$, $I(t)>0$, and $V(t) >0$, for all $t>0$.

Our next goal is to demonstrate that the solutions of the system remain bounded. Consider the quantity $W(t) = T(t)+E(t)+I(t)+\frac{\delta + d(t)}{2p} V(t)$. The derivative of $W$ can be computed as follows:
\begin{align*}
\dot{W} &= \dot{T}+\dot{E}+\dot{I}+\frac{\delta + d(t)}{2p} \dot{V} \\
&= \mu(t)-d(t) T-d(t) E-\frac{\delta + d(t)}{2} I- c V \\
&\leq \mu(t)-m\left(T+E+I+\frac{\delta + d(t)}{2p} V\right) \leq \zeta -m W,
\end{align*}
where $\zeta = \max_{t \in[0, T_{M})} \mu(t)$ and  $m$ is defined as $m=\min\{\min_{t \in[0, T_{M})}, \min_{t \in[0, T_{M})} \frac{\delta + d(t)}{2p}, c\}$. The inequality can be rewritten as $\dot{W}+m W \leq \zeta$, which forms a linear ordinary differential inequality with a negative coefficient for $W$. Applying the integrating factor $e^{mt}$, we obtain:
$$
W(t) \leq e^{-mt}\left(W(0)-\frac{\zeta}{m}\right)+\frac{\zeta}{m}.
$$
This implies that $0 \leq W(t) \leq \xi$, where $\xi=\frac{\zeta}{m}$.  Consequently, we have $0 \leq T(t), E(t), I(t) \leq \xi$ and $0 \leq V(t) \leq N_1$ for all $t \geq 0$, provided that the initial conditions satisfy $T(0)+E(0)+I(0)+\frac{\delta + d(0)}{2p} V(0) \leq \xi$. Here, we define $N_1=\frac{2p \xi}{\delta+D}$. In summary, we have successfully demonstrated that the solutions are bounded. Therefore, we can state that the set
$$
\mathcal{D} := \left\{ (T, E, I, V) \in \mathbb{R}^4_+ \mid T + E + I \leq 3 \xi, \; V \leq N_1 \right\}
$$
is a positively invariant, compact and attractor.
\end{proof}
\section{Existence and uniqueness of the virus-free $\omega$-periodic solution}
\begin{Lemma}
\label{Lemma1}
The system $\eqref{1}$ admits a unique virus-free periodic solution given by $\mathcal{E}_0=\left(T^*(t), 0,0,0\right)$, where $T^*$ is the unique periodic solution of \eqref{2}. This solution has a period $\omega$.
\end{Lemma}
\begin{proof}
    From  system \eqref{1}, setting $V = 0$, we have the  equation
    \begin{equation}
    \label{2}
\frac{d T}{d t}=\mu(t)-d(t) T
\end{equation}
with initial condition $T(0)=T_0 \in \mathbb{R}_{+}$.

From \eqref{2}, we have 
$$ \frac{d T}{d t} + d(t) T =\mu(t) .$$
Multiplying by $e^{\int_0^t d(\xi) d\xi}$, we have
$$ \frac{d T}{d t} e^{\int_0^t d(\xi) d\xi} + d(t) T e^{\int_0^t d(\xi) d\xi} =\mu(t)e^{\int_0^t d(\xi) d\xi}. $$
After integrating both sides we get
$$ T(t) = e^{-\int_0^t d(\xi) d\xi} \left( \int_0^t \mu(\tau)e^{\int_0^\tau d(\xi) d\xi} + T_0 \right).$$
To find the virus-free periodic solution $T^*(t)$, we must ensure $T^*(t+\omega)=T^*(t)$, from this, one gets the $\omega$-periodic solution
$$ T^*(t) = e^{-\int_0^t d(\xi) d\xi} \left( \int_0^t \mu(\tau)e^{\int_0^\tau d(\xi) d\xi} + \frac{e^{-\int_0^\omega d(\xi) d\xi} \int_0^\omega \mu(\tau)e^{\int_0^\tau d(\xi) d\xi} d\tau}{1-e^{-\int_0^\omega d(\xi) d\xi}} \right).  $$
\end{proof}
\section{The basic reproduction ratio}
In the preceding section, we introduced our heterogeneous epidemiological model for periodic environments. Now, our focus shifts to a specific aspect of its analysis—calculating the basic reproduction ratio, $\mathcal{R}_0$, and studying the stability of the infection-free equilibrium. To calculate $\mathcal{R}_0$, we employ the technique outlined in \cite{Wendi}. Subsequently, we extend our investigation to assess the stability of the periodic infection-free steady state, utilizing the same methodology. 

We define $\mathcal{F}(t, x)$ as the input rate of newly infected individuals in the $i$-th compartment, $\mathcal{V}^+(t, x)$ as the input rate of individuals through other means, and $\mathcal{V}^-(t, x)$ as the rate of transfer of individuals out of compartment $i$. Specifically:
$$
\mathcal{F}(t, x) = \begin{bmatrix}
    \frac{\beta(t) T V}{(1+C_1 T)(1+C_2 V)} \\
    0 \\
    0 \\
    0 \\
\end{bmatrix},
$$
$$
\mathcal{V}^+(t, x) = \begin{bmatrix}
    0 \\
    k E \\
    p I \\
    \mu(t) \\
\end{bmatrix}, \qquad \mathcal{V}^-(t, x) = \begin{bmatrix}
    (k+d(t)) E\\
    (\delta+d(t)) I \\
    c V \\
    \frac{\beta(t) T V}{(1+C_1 T)(1+C_2 V)}+d(t) T 
\end{bmatrix}
$$
Let $x = (E, I, V, T)$, where $E, I$ and $V$ are the infection-related compartments and $T$ is the uninfected compartment. Using the same notation as \cite{Wendi}, system $\eqref{1}$ can be written as
$$
\frac{d x_i}{d t}=\mathcal{F}_i(t, x)-(\mathcal{V}^-(t, x) - \mathcal{V}^+_i(t, x)) \triangleq f_i(t, x), \quad i=1, \ldots, 4.
$$ It is easy to verify that system $\eqref{1}$ satisfies assumptions (A1)--(A5) in \cite{Wendi}.

Letting $x^* = (0, 0, 0, T^*(t))$ be the unique positive $ \omega$-periodic solution of the model \eqref{1}, set
\begin{align*}
    F(t)&=\left(\frac{\partial \mathcal{F}_i\left(t, x^*\right)}{\partial x_j}\right)_{1 \leq i, j \leq 3} =\begin{bmatrix}
0 & 0 & \beta(t) \frac{T^*(t)}{1+C_1 T^*(t)} \\
0 & 0 & 0 \\
0 & 0 & 0
\end{bmatrix},\\
 G(t)&=\left(\frac{\partial \mathcal{V}_i\left(t, x^*\right)}{\partial x_j}\right)_{1 \leq i, j \leq 3} =\begin{bmatrix}
\kappa + d(t) & 0 & 0 \\
-\kappa & \delta +d(t) & 0 \\
0 & -p & c
\end{bmatrix}.
\end{align*}

Let $\Phi_M(t)$ and $\rho\left(\Phi_M(\omega)\right)$ be the monodromy matrix of the linear $\omega$-periodic system $\frac{d z(t)}{d t}=M(t) z$ and the spectral radius of $\Phi_M(\omega)$, respectively. Let $X(t, s), t \geq s$, be the evolution operator of the linear $\omega$-periodic system
$$
\frac{d y}{d t}=-G(t) y,
$$
that is, for each $s \in \mathbb{R}$, the $3 \times 3$ matrix $X(t, s)$ satisfies
$$
\frac{d X(t, s)}{d t}=-G(t) X(t, s), \quad \forall t \geq s, \quad X(s, s)=I,
$$
where $I$ is the $3 \times 3$ identity matrix. Thus, the monodromy matrix $\Phi_{-G}(t)$ of system \eqref{(3)} equals to $X(t, 0),\ t \geq 0$. 

Let $\psi(s)$ denote the distribution of infected individuals, which exhibits $\omega$-periodicity in $s$. The expression $F(s) \psi(s)$ signifies the rate of new cases arising from individuals infected at time $s$. Beyond this, for $t \geq s$, the term $X(t, s) F(s) \psi(s)$ characterizes the distribution of individuals who contracted the infection at time $s$ and remain infectious at time $t$. Consequently, 
\begin{equation}\label{(3)}
g(t) \coloneqq \int_{-\infty}^t X(t, s) F(s) \psi(s) ds = \int_0^{\infty} X(t, t-a) F(t-a) \psi(t-a) da.
\end{equation}

The function $g(t)$ represents the distribution of cumulative new infections at time $t$, generated by all infected individuals $\psi(s)$ introduced at any time $s \leq t$.

Denote by $C_\omega$ the ordered Banach space of $\omega$-periodic functions from $\mathbb{R}$ to $\mathbb{R}^6$, equipped with the usual maximum norm $\|\cdot\|_{\infty}$. Define the positive cone
\begin{equation}
C_\omega^{+} \coloneqq \left\{\psi \in C_\omega : \psi(t) \geq 0, \forall t \in \mathbb{R}\right\}.
\end{equation}

The linear next infection operator $L\colon  C_\omega \rightarrow C_\omega$ is introduced as
\begin{equation}
(L \psi)(t) = \int_0^{\infty} X(t, t-a) F(t-a) \psi(t-a) da, \quad \forall t \in \mathbb{R}, \psi \in C_\omega.
\end{equation}

The basic reproduction number of \eqref{1} is defined as $\mathcal{R}_0 \coloneqq \rho(L)$, the spectral radius of the operator $L$. For the autonomous case, as derived in \cite{Nali}, this number is given by
$$\mathcal{R}_0 = \frac{p \beta k \mu}{c (d + \delta)(d + k)(d + C_1 \mu)}.
$$
Furthermore, consider the system
 described by the differential equation
\begin{equation}
\frac{d w}{d t} = \left(-G(t) + \frac{1}{\lambda} F(t)\right) w, \quad t \in \mathbb{R},
\end{equation}
where the parameter $\lambda$ is in the interval $(0, \infty)$. Since $F(t)$ is nonnegative and $-G(t)$ is cooperative, it follows that $\rho(W(\omega, \lambda))$ is a continuous and nonincreasing function with respect to $\lambda$ in the interval $(0, \infty)$, and moreover, $\lim_{\lambda \to \infty} \rho(W(\omega, \lambda)) < 1$. These conditions ensure that assumption (A7) is satisfied, which leads to the following two results.

To express the basic reproduction number for the system \eqref{1}, observe that $W(t, \lambda) = \Phi_{(F / \lambda) - G}(t)$ and $\Phi_{F - G}(t) = W(t, 1)$ for all $t \geq 0$. According to Theorem 2.1 in \cite{Wendi}, the basic reproduction number can be defined as $\lambda_0$ such that $\rho(\Phi_{(F / \lambda_0) - V}(\omega)) = 1$. This value of $\lambda_0$ can be straightforwardly calculated.
\section{Stability analysis}
To analyze the stability of the virus-free periodic solution $ \mathcal{E}_0$, we invoke a pivotal theorem from \cite{Wendi} and a lemma from \cite{Nakata} to support our proof.
\begin{theorem}[{\cite[Theorem 2.2]{Wendi}}]
\label{thm22}\leavevmode
\begin{itemize}
   \item[(i)] $\mathcal{R}_0=1$ if and only if $\rho\left(\Phi_{F-V}(\omega)\right)=1$.
\item[(ii)] $\mathcal{R}_0>1$ if and only if $\rho\left(\Phi_{F-V}(\omega)\right)>1$.
\item[(iii)] $\mathcal{R}_0<1$ if and only if $\rho\left(\Phi_{F-V}(\omega)\right)<1$.
\end{itemize}
\end{theorem}
\begin{Lemma}
\label{Lemma2}
Let $A(t)$ be a continuous, cooperative, irreducible and $\omega$-periodic matrix function, let $\Phi_A(t)$ be the fundamental matrix solution of
\begin{equation*}
\label{eqq}
x^{\prime}=A(t) x
\end{equation*}
and let $\sigma=\frac{1}{\omega} \ln \left(\rho\left(\Phi_A(\omega)\right)\right)$, where $\rho$ denotes the spectral radius. Then, there exists a positive $\omega$-periodic function $v(t)$ such that $\mathrm{e}^{\sigma t} v(t)$ is a solution of \eqref{eqq}.
\end{Lemma}
\subsection{Local and global stability of the virus-free periodic solution}

\begin{theorem}
   If $\mathcal{R}_0 < 1$, the virus-free periodic solution $\mathcal{E}_0$ is both locally and globally asymptotically stable in $\mathbb{R}^4_+$
\end{theorem}
\begin{proof}
  To establish that $\mathcal{E}_0$ is locally asymptotically stable when $\mathcal{R}_0 < 1$, we apply two fundamental theorems from \cite{Wendi}. Additionally, we need to show that $\mathcal{E}_0$ is globally attractive in $\mathbb{R}^3_+$ under the condition $\mathcal{R}_0 < 1$.

  According to Lemma \ref{Lemma1}, for any $\epsilon > 0$, there exists a time $T_1 > 0$ such that $T(t) \leq T^*(t) + \epsilon$ for all $t > T_1$. Since the function $ f(x) = \frac{x}{1+c x} $ is increasing over~$\mathbb{R}$, it follows that
  $$
  \begin{aligned}
    \frac{\beta(t) T(t) V(t)}{(1+C_1 T(t))(1+C_2 V(t))} &\leq \frac{\beta(t) T(t)}{1+C_1 T(t)} \leq \beta(t) \frac{T^*(t) + \epsilon}{1+ C T^*(t)},
  \end{aligned}
  $$
  for $ t > T_1 $. Consequently, from system \eqref{1}, we obtain
  $$
  \begin{cases}
  E^{\prime} \leq \beta(t) \frac{T^*(t) + \epsilon}{1+ C T^*(t)} - (k + d(t)) E, \\
  I^{\prime} = k E - (\delta + d(t)) I, \\
  V^{\prime} = p I - c V.
  \end{cases}
    $$

  Define the matrix
  $$
  M_1(t) = 
  \begin{bmatrix}
  0 & 0 & \beta(t) \\
  0 & 0 & 0 \\
  0 & 0 & 0
  \end{bmatrix}
  .
  $$

  Using Theorem \ref{thm22}, we deduce that $\rho\left(\Phi_{F-G}(\omega)\right) < 1$. Select $\epsilon > 0$ such that $\rho\left(\Phi_{F-G+\varphi(\epsilon) M_1}(\omega)\right) < 1$ and consider the modified system
  $$
  \begin{bmatrix}
  \Bar{E}^{\prime} \\
  \Bar{I}^{\prime} \\
  \Bar{V}^{\prime}
  \end{bmatrix}
 = \left(F(t) - G(t) + \varphi(\epsilon) M_1(t)\right) 
  \begin{bmatrix}
  \Bar{E} \\
  \Bar{I} \\
  \Bar{V}
  \end{bmatrix}
.
  $$
By applying Lemma \ref{Lemma2} and the standard comparison principle, we find $\omega$-periodic functions $v_1(t)$, $v_2(t)$ and $v_3(t)$ such that
  $$
  E(t) \leq v_1(t) e^{\sigma t}, \quad I(t) \leq v_2(t) e^{\sigma t}, \quad \text{and} \quad V(t) \leq v_3(t) e^{\sigma t},
  $$
  where $\sigma = \frac{1}{\omega} \ln \left(\rho\left(\Phi_{F-G+\varphi(\epsilon) M_1}(\omega)\right)\right)$. This implies that $ E(t) \rightarrow 0 $, $ I(t) \rightarrow 0 $, and $ V(t) \rightarrow 0$ as $ t \rightarrow \infty $.

  Therefore, we have
  $$
  \lim_{t \rightarrow \infty} \left(T(t) - T^*(t)\right) = 0.
  $$
  The proof is now complete.
\end{proof}
\subsection{Persistence of the infective compartment and existence of endemic periodic orbits}
Let $ u(t, X_0) $ be the solution of system \eqref{1} with the initial condition $ X_0 = (T^0, E^0, I^0, V^0) $. By the fundamental theorem of existence and uniqueness \cite{Perko}, the solution $ u(t, X_0) $ is unique, with $ u(0, X_0) = X_0 $.

Define the Poincaré map $ Q \colon \mathbb{R}^4_+ \rightarrow \mathbb{R}^4_+ $ associated with system \eqref{1} as $ Q(X_0) = u(\omega, X_0) $ for all $ X_0 \in \mathbb{R}^4_+ $. Consider the set $\Gamma = \left\{ (T, E, I, V) \in \mathbb{R}_+^4 \right\}$, its interior $\Gamma_0 = \operatorname{Int}(\mathbb{R}_+^4)$, and its boundary $\partial \Gamma_0 = \Gamma \setminus \Gamma_0$. Note that both $\Gamma$ and $\Gamma_0$ are positively invariant under the dynamics of the system, and that $Q$ is point dissipative.

Now, define the set
$$
M_{\partial} = \left\{\left(T^0, E^0, I^0, V^0\right) \in \partial \Gamma_0 : Q^n\left(T^0, E^0, I^0, V^0\right) \in \partial \Gamma_0, \, \text{for all } n \geq 0\right\}.
$$
To apply the theory of uniform persistence as outlined in \cite{Zhao}, we need to prove that 
\begin{equation}
\label{equa}
     M_{\partial} = \left\{(T, 0, 0, 0) \in \mathbb{R}^4_+ : T \geq 0 \right\}.
\end{equation}

Note that  $M_{\partial} \supseteq\{(T,0, 0,0) \in \mathbb{R}^4_+: T \geq 0\}$, to show that $M_{\partial} \subseteq\{(T,0, 0,0) \in \mathbb{R}^4_+: T \geq 0\}$ We begin by using a contradiction argument. Assume that $ M_{\partial} $ is not a subset of $ \{(T, 0, 0, 0) \in \mathbb{R}^4_+ : T \geq 0\} $. This implies there exists an initial condition $ (T(0), E(0), I(0), V(0)) \in \partial \Gamma_0 $ such that for some $ m_1 \geq 0 $, the solution satisfies $ \left(E\left(m_1 \omega\right), I\left(m_1 \omega\right), V\left(m_1 \omega\right)\right)^T > 0 $.

Now, consider $ m_1 \omega $ as the initial time. From the positivity of the solutions discussed earlier, we have $ \left(T(t), E(t), I(t), V(t)\right) > 0 $ for all $ t > m_1 \omega $. This implies that $ \left(T(t), E(t), I(t), V(t)\right) \in \Gamma_0 $ for all $ t > m_1 \omega $, meaning the solution enters the interior of $\Gamma$, contradicting the assumption that it remains on $ \partial \Gamma_0 $. 

Thus, our assumption must be false, and it follows that $ M_{\partial} \subseteq \{(T, 0, 0, 0) \in \mathbb{R}^4_+ : T \geq 0\} $, thereby proving that equality \eqref{equa} holds. We deduce, therefore, the uniform persistence of the disease as follows.
\begin{theorem}
    Suppose that $\mathcal{R}_0 > 1$. Then, system \eqref{1} has at least one positive periodic solution. Moreover, there exists a constant $\eta > 0$ such that for any initial condition $\left(T^0, E^0, I^0, V^0\right) \in \Gamma_0$, we have $$\liminf_{t \rightarrow \infty} (E(t), I(t), V(t)) \geq (\eta, \eta, \eta).$$
\end{theorem}
\begin{proof}
    The proof follows a similar approach to that presented in \cite{Almuashi,Miled}.  We start by proving that the trajectory of the system \eqref{1} is uniformly persistent with respect to $\left(\Gamma_0, \partial \Gamma_0\right)$ by using (Theorem 3.1.1 in \cite{Zhao}). Recall that using theorem \eqref{thm22}, we obtain that $\rho\left(\varphi_{F-G}(T)\right)>1$. Therefore, there exists $\nu>0$ small enough and satisfying that $r\left(\varphi_{F-G-\nu M_1}(T)\right)>1$. Let us consider the following perturbed equation.
    \begin{equation}
    \label{perteq}
        \frac{d T_{\alpha}}{d t}= \mu(t) - \alpha\frac{\beta(t) T_{\alpha} }{(1+C_1 T_{\alpha})(1+C_2 \alpha)}- d(t) T_{\alpha}
    \end{equation}
    The function $Q$ associated with the perturbed equation \eqref{perteq} has a unique positive fixed point $\left(\bar{T}_\alpha^0\right)$ that it is globally attractive in $\mathbb{R}_{+}$. We apply the implicit function theorem to deduce that $\left(\bar{T}_\alpha^0\right)$ is continuous with respect to $\alpha$. Therefore, we can choose $\alpha>0$ small enough and satisfying $\bar{T}_\alpha(t)>$ $\bar{T}(t)-\nu$ for all $t>0$. Let $M_1=\left(\bar{T}^0, 0,0, 0\right)$. Since the trajectory is continuous with respect to the initial condition, there exists $\alpha^*$ satisfying $\left(T^0, E^0, I^0, V^0\right) \in \Gamma_0$ with $\|\left(T^0, E^0, I^0, V^0\right)-$ $u\left(t, M_1\right) \| \leq \alpha^*$; it holds that

$$
\left\|u\left(t,\left(T^0, E^0, I^0, V^0\right)\right)-u\left(t, M_1\right)\right\|<\alpha \text {\quad for } 0 \leq t \leq T
$$
We prove by contradiction that
\begin{equation}
\label{eqsup}
\limsup _{n \rightarrow \infty} d\left(P^n\left(T^0, E^0, I^0, V^0\right), M_1\right) \geq \alpha^* \qquad\forall\left(T^0, E^0, I^0, V^0\right) \in \Gamma_0
\end{equation}
Suppose that $\limsup _{n \rightarrow \infty} d\left(Q^n\left(T^0, E^0, I^0, V^0\right), M_1\right)<\alpha^*$ for some $\left(T^0, E^0, I^0, V^0\right) \in \Gamma_0$. We can assume that $d\left(Q^n\left(T^0, E^0, I^0, V^0\right), M_1\right)<\alpha^*$ for all $n>0$. Therefore
$$
\left\|u\left(t, Q^n\left(T^0, E^0, I^0, V^0\right)\right)-u\left(t, M_1\right)\right\|<\alpha\qquad \forall n>0 \text { and } 0 \leq t \leq \omega . 
$$
For all $t \geq 0$, let $t=n \omega+t_1$, with $t_1 \in[0, \omega)$ and $n=\big[\frac{t}{\omega}\big]$ (i.e.\ the greatest integer $\leq \frac{t}{\omega}$). Then, we get
$$
\left\|u\left(t,\left(T^0, E^0, I^0, V^0\right)\right)-u\left(t, M_1\right)\right\|=\left\|u\left(t_1, P^n\left(T^0, E^0, I^0, V^0\right)\right)-u\left(t_1, M_1\right)\right\|<\alpha \text { for all } t \geq 0
$$
Set $(T(t), E(t), I(t), V(t))=u\left(t,\left(T^0, E^0, I^0, V^0\right)\right)$. Therefore $0 \leq E(t), I(t),V(t) \leq \alpha, t \geq 0$ and
\end{proof}
 \begin{equation}
    \label{perteq}
        \frac{d T_{\alpha}}{d t} \geq \mu(t) - \alpha\frac{\beta(t) T }{(1+C_1 T)(1+C_2 \alpha)}- d(t) T
    \end{equation}
The fixed point $\bar{T}_\alpha^0$ of the function $Q$ associated with the perturbed system \eqref{perteq} is globally attractive such that $\bar{T}_\alpha(t)>\bar{T}(t)-\nu$, then, there exists $T_2>0$ large enough and satisfying the condition that $T(t)>\bar{T}(t)-\nu$. Therefore, for $t>T_2$,
$$
\begin{dcases}\frac{d E}{d t} \geq\frac{\beta(t) (\bar{T}(t)-\nu) V}{(1+C_1  (\bar{T}(t)-\nu))(1+C_2 V)}-k E - d(t) E, \\
 \frac{d I}{d t}=k E-\delta I - d(t) I, \\
 \frac{d V}{d t}=p I-c V.\end{dcases}
$$
Note that we have the condition that $\rho\left(\varphi_{F-G-\nu M_2}(\omega)\right)>1$. Applying \eqref{Lemma2} and the comparison principle, there exists a positive $\omega$-periodic trajectory $y_2(t)$ satisfying the condition that $J(t) \geq e^{k_2 t} y_2(t)$ with $k_2=\frac{1}{\omega} \ln \rho\left(\varphi_{F-G-\nu M_2}(\omega)\right)>0$, which implies that $\lim _{t \rightarrow \infty} E(t)=\infty$, $\lim _{t \rightarrow \infty} I(t)=\infty$ and $\lim _{t \rightarrow \infty} V(t)=\infty$ which is impossible since the trajectories are bounded. Therefore, the inequality \eqref{eqsup} is satisfied and $Q$ is weakly uniformly persistent with respect to $\left(\Gamma_0, \partial \Gamma_0\right)$. From Lemma \eqref{Latr} $ P$ has a global attractor. It follows that $M_1$ is an isolated invariant set in $X$ and $W^s\left(M_1\right) \cap \Gamma_0=\emptyset$. It is clear that every solution in $M_{\partial}$ converges to $M_1$ and $M_1$ is acyclic in $M_{\partial}$. From (Theorem 1.3.1 and Remark 1.3.1 in \cite{Zhao}), we obtain that $Q$ is uniformly persistent with respect to $\left(\Gamma_0, \partial \Gamma_0\right)$. Moreover, according to (Theorem 1.3.6 in \cite{Zhao}), $Q$ has a fixed point $\left(\tilde{T}^0, \tilde{E}^0, \tilde{I}^0, \tilde{V}^0\right) \in \Gamma_0$. We can see that $$\left(\tilde{T}^0, \tilde{E}^0, \tilde{I}^0, \tilde{V}^0\right) \in R_{+} \times \operatorname{Int}\left(R_{+}^3\right)$$
We prove also by contradiction that $\tilde{T}^0 > 0$. Assume that $\tilde{T}^0 = 0$. Using the first equation of the system $\eqref{1}$, verifies that 
$$ \dot{\tilde{T}}(t) \geq \mu(t) -\frac{\beta(t) \tilde{T}(t) \tilde{V}(t)}{(1+C_1 \tilde{T}(t))(1+C_2 \tilde{V}(t))}- d(t) \tilde{T}(t)$$
with $\tilde{T}^0=\tilde{T}(l \omega)=0, l=1,2,3, \ldots$ By Lemma \eqref{Latr} , for any $r>0$, there exists a sufficiently large enough $T_3>0$, such that $\tilde{V}(t) \leq N_1+r for t>T_3$,  Additionally, it is straightforward to see that $x/(1+cx) \leq x $.  Then, we have
$$ \dot{\tilde{T}}(t) \geq \mu(t) - \left( \beta(t) ( N_1 + r ) + d(t) \right) \tilde{T}(t), \quad \text{for} \; t \geq T_3 $$
There exists a large $\bar{l}$ such that $l \omega>T_3$ for $l>\bar{l}$. By the comparison principle, we obtain

$$
\begin{aligned}
\tilde{T}(l \omega)= & e^{-\int_0^{l \omega}\left(\beta(u) (N_1 +r) + d(u)\right) d u} \\
& \times\left(\tilde{T}^0+\int_0^{l \omega} \mu (s)  e^{\int_0^s\left(\beta(u) (N_1 +r) + d(u)\right) d u} ds \right)>0
\end{aligned}
$$

for any $l>\bar{l}$. Then, we see a contradiction. Thus, $\tilde{T}^0>0$ and $\left(\tilde{T}^0, \tilde{E}^0, \tilde{I}^0, \tilde{V}^0\right)$ is a positive $\omega$-periodic solution of $\eqref{1}$. Thus, the proof is complete.
\section{Numerical simulation}
For the numerical simulations, we use circadian rhythms as an example, which represent biological processes following approximately 24-hour cycles. To incorporate these periodic fluctuations into our model, we define the relevant periodic parameters using sine functions:  
$$
\begin{cases}
\mu(t) = \mu_0 + \mu_1 \sin(\omega t), \\
\beta(t) = \beta_0 + \beta_1 \sin(\omega t), \\
d(t) = d_0 + d_1 \sin(\omega t).
\end{cases}
$$  
Sine functions are particularly suited for modeling periodic behaviors in processes that start from a baseline level rather than a peak, as sine waves naturally begin at zero. This choice aligns well with the dynamics of circadian rhythms, where many biological processes exhibit oscillations that initiate from a neutral or equilibrium state. While other periodic forms, such as cosine-based functions, have been widely used in the literature \cite{Miled, Almuashi}, we selected sine functions to better represent the specific phase characteristics of the oscillatory processes considered in our model.

The fixed constants used for the numerical simulations are presented in Table~\eqref{tabN1}. These parameter values were carefully chosen to validate the theoretical results of our model.
. These include the base parameters $\mu_0,\beta_0,d_0$ and their respective amplitudes $\mu_1,\beta_1,d_1$. The angular frequency $\omega = \frac{2 \pi}{24}$ corresponds to the period of the circadian rhythm. \\$ $\\
\begin{table}[H] 
\centering
\caption{Fixed constants used for numerical simulation} \label{tabN1}
\begin{tabular}{|c|c|c|c|c|c|c|c|}
\hline
Parameter & $\mu_0$ & $\beta_0$ & $d_0$ & $\mu_1$ & $\beta_1$ & $d_1$ & $\omega$ \\
\hline
Value & 0.1 & 0.3 & 0.01 & 0.05 & 0.1 & 0.005 & $\frac{2 \pi}{24}$ \\
\hline
\end{tabular}
\end{table}
We consider two distinct cases to illustrate the asymptotic behavior of the solution to the model given by equation \eqref{1}. In the first case, shown in Figure \eqref{fig1}, where the basic reproduction number satisfies the condition $\mathcal{R}_0 < 1$, the approximate solution of the model \eqref{1} converges to the virus-free periodic trajectory $\mathcal{E}_0 = \left( T^*(t), 0, 0, 0 \right)$. This outcome indicates that when $\mathcal{R}_0 < 1$, the infection dies out over time, leading to a stable periodic state characterized by the absence of the virus. In contrast, in the second case, depicted in Figure \eqref{fig2}, where the basic reproduction number satisfies $\mathcal{R}_0 > 1$, the approximate solution of the model \eqref{1} exhibits a different asymptotic behavior. In this scenario, the solution does not converge to the virus-free state; instead, it asymptotically approaches a periodic solution where the infection persists over time. This indicates that when $\mathcal{R}_0 > 1$, the virus maintains a presence in the population, and the system reaches a new periodic state in which the infection continues to exist. These two cases, therefore, demonstrate the critical threshold effect dictated by the value of $\mathcal{R}_0$, highlighting the conditions under which the infection either dies out or persists within the host population.
\begin{figure}[H]
    \centering
    \includegraphics[width=\textwidth]{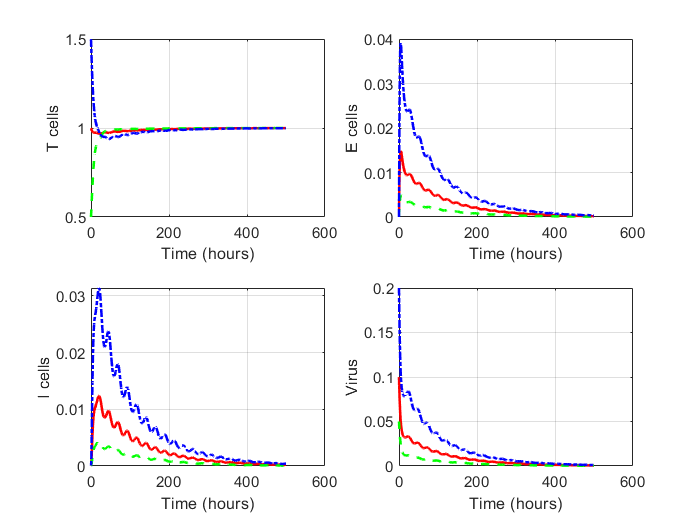}
    \caption{Time series of the system \eqref{1} under different initial conditions, with $\mathcal{R}_0 < 1$. The plots show the dynamics of T cells, E cells, I cells, and virus concentration over a 10-day period (240 hours). The constant parameters used in the model are: $C_1 = 0.1$, $C_2 = 0.1$, $k = 0.2$, $\delta = 0.09$, $p = 0.5$, and $c = 0.18$.}

    \label{fig1}
\end{figure}
\begin{figure}[H]
    \centering
    \includegraphics[width=\textwidth]{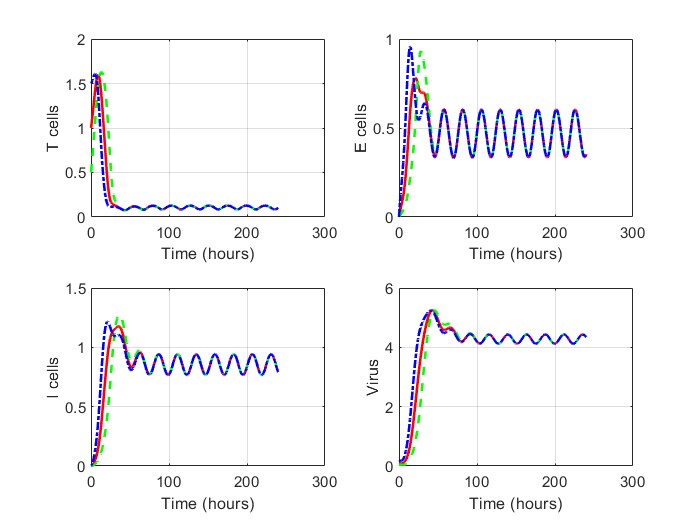}
    \caption{Time series of the system \eqref{1} under different initial conditions, with $\mathcal{R}_0 > 1$. The plots show the dynamics of T cells, E cells, I cells, and virus concentration over a 10-day period (240 hours). The constant parameters used in the model are: $C_1 = 0.1$, $C_2 = 0.1$, $k = 0.2$, $\delta = 0.1$, $p = 0.5$, and $c = 0.1$.}

    \label{fig2}
\end{figure}
In Figure \eqref{fig:phase_planes}, we provide an enlarged view of the limit cycle corresponding to the case where $\mathcal{R}_0 > 1$. This magnification allows for a clearer visualization of the system's dynamics when the infection persists. The figure illustrates the periodic oscillations that characterize the limit cycle behavior, emphasizing how the trajectories converge to a stable periodic orbit.

\begin{figure}[H]
    \centering
    \begin{minipage}[b]{0.3\textwidth}
        \centering
        \includegraphics[width=\textwidth]{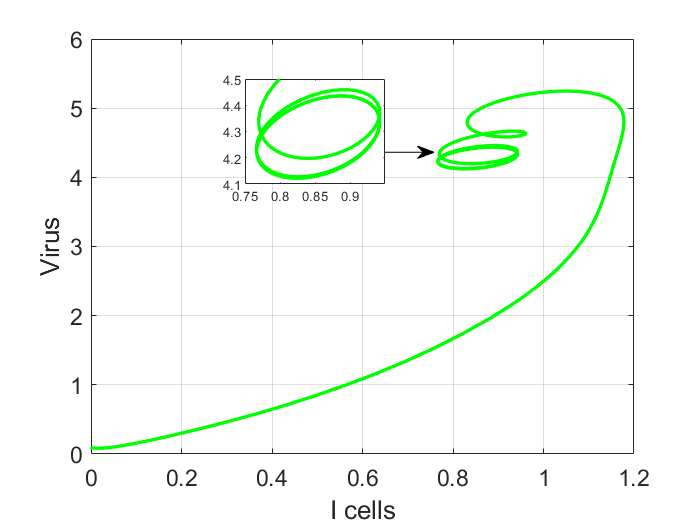}
        \label{fig:I_V}
    \end{minipage}
    \hfill
    \begin{minipage}[b]{0.3\textwidth}
        \centering
        \includegraphics[width=\textwidth]{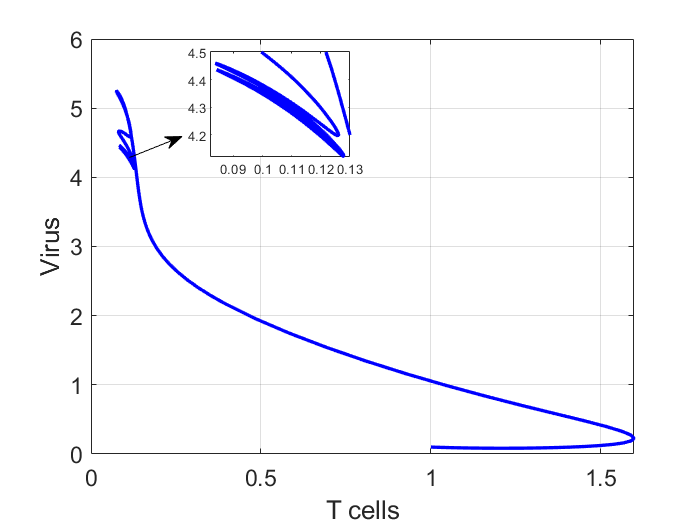}
        \label{fig:T_V}
    \end{minipage}
    \hfill
    \begin{minipage}[b]{0.3\textwidth}
        \centering
        \includegraphics[width=\textwidth]{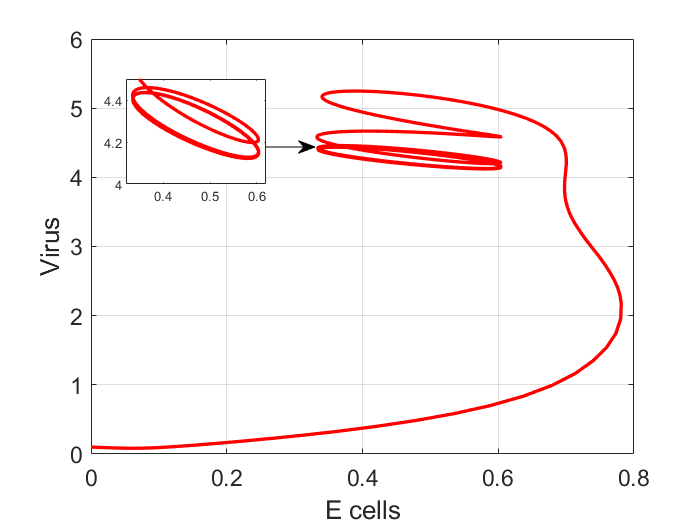}
        \label{fig:E_V}
    \end{minipage}
    \caption{Magnified view of the limit cycle of the dynamics \eqref{1}}\label{fig:phase_planes}
\end{figure}
\section{Discussion and conclusion}
In this study, we have developed a within-host infection model that incorporates periodic effects by introducing time-dependent coefficients to capture the dynamic nature of biological processes influenced by circadian rhythms. The core focus of our analysis has been to understand the threshold dynamics of the model, which are primarily governed by the basic reproduction number, $\mathcal{R}_0$. Our results demonstrate that the global behavior of the infection is decisively determined by this threshold parameter.

When $\mathcal{R}_0 < 1$, we have established that the global asymptotic stability of the virus-free periodic solution, $\mathcal{E}_0$, is guaranteed. This finding implies that under conditions where the basic reproduction number is less than unity, the infection will eventually be cleared from the host, regardless of the initial conditions. Conversely, when $\mathcal{R}_0 > 1$, the model reveals that the infection persists within the host, evolving towards a stable periodic state characterized by sustained viral presence.

A significant aspect of our model is the incorporation of the Crowley--Martin functional response to describe the infection rate. Unlike simpler mass-action, the Crowley--Martin response provides a more realistic representation of the host-pathogen interaction by accounting for the saturation effect of susceptible cells in the presence of an increasing viral load.

Our numerical simulations further support these theoretical findings. They indicate that, for $\mathcal{R}_0 > 1$, there exists a unique positive periodic solution that is globally asymptotically stable. This result implies that when the reproduction number exceeds unity, the infection dynamics do not merely oscillate randomly; instead, they settle into a predictable and recurring pattern over time. This periodic behavior highlights the importance of considering time-dependent factors in within-host models, as they can significantly affect the infection's persistence and overall progression.

In conclusion, our work provides new insights into the role of periodic factors in the dynamics of within-host infection models and emphasizes the utility of the Crowley--Martin functional response in capturing complex biological interactions. However, to further validate and refine our model, future work will focus on integrating real-world data to enhance its applicability and accuracy in describing actual infection scenarios. This data-driven approach will help us better understand the model's potential in clinical and epidemiological settings.
\section*{Acknowledgements}
This research was supported by project TKP2021-NVA-09, implemented with the support provided by the Ministry of Innovation and Technology of Hungary from the National Research, Development and Innovation Fund, financed under the TKP2021-NVA funding~scheme. I.N. was supported
by ÚNKP-23-3-New National Excellence Program of the Ministry for Innovation and Technology from the source of the National Research, Development and Innovation Fund and Stipendium Hungaricum scholarship with Application No.~403679. A.D. was supported by the National Laboratory for Health Security,
RRF-2.3.1-21-2022-00006 and by the project No.~129877,
implemented with the support provided from the National Research, Development and Innovation Fund of Hungary, financed under the KKP\_19 funding
scheme.
\section*{Conflicts of interest}
The authors declare that they have no conflict of interest related to this
research. The~authors have no financial interests in any companies or organizations that could benefit from the results of this study. This research was conducted independently and without any influence from any third~parties.

\end{document}